\documentclass[10pt,fleqn]{article}

\usepackage[english]{babel}
\usepackage[utf8]{inputenc}     
\usepackage{comment}
\usepackage{a4wide}
\usepackage{amsfonts}
\usepackage{amsthm}
\usepackage{amssymb}

\usepackage[fleqn]{amsmath}

\usepackage{enumerate}
\usepackage{enumitem}
\usepackage{mathtools}
\usepackage{bbm}
\usepackage{dsfont}
\usepackage{ifthen, enumitem, mathrsfs, euscript,anyfontsize}
\usepackage{graphicx}
\usepackage[table]{xcolor}
\usepackage{listings}
\usepackage{url}
\usepackage{amsfonts}
\usepackage{enumerate}
\usepackage{blindtext}
\usepackage{hyperref}
\usepackage{xcolor}
 \everymath{\displaystyle}

\usepackage{thmtools,thm-restate}

\usepackage{url}
\usepackage{tikz}
\usepackage{multirow}

\newtheorem{theorem}{Theorem}
\newtheorem*{definition}{Definition}
\newtheorem{lemma}{Lemma}

\newtheorem{proposition}{Proposition}

\everymath{\displaystyle }
\allowdisplaybreaks

\newcommand{\N}{\mathbb N}  
\newcommand{\R}{\mathbb R}

\newcommand{\J}{\mathcal{J}}

\newcommand{\E}{{\mathds{E}}}
\newcommand{\prob}{\mathds{P}}
\newcommand{\abs}[1]{\left| #1 \right|} 

\newcommand{\norm}[1]{\| #1 \|} 

\newcommand{\pare}[1]{\left( #1 \right)} 
\newcommand{\set}[1]{{\left\{ #1 \right\}}} 
\newcommand{\corch}[1]{\left[ #1 \right]} 

\newcommand{\tq}{\;\ifnum\currentgrouptype=16 \middle\fi|\;}
\newcommand{\ts}{\textsuperscript}

 \makeatletter

\def\blfootnote{\gdef\@thefnmark{}\@footnotetext}

\makeatother

\title{
Poisson genericity in  numeration systems \\with exponentially mixing probabilities 
\blfootnote{This manuscript has been accepted for publication in  \emph{Transactions of the American Mathematical Society}.}}

 \author{
 Nicolás Álvarez \and
 Verónica Becher \and
 Eda Cesaratto \and
 Martín Mereb \and
 Yuval Peres \and
 Benjamin Weiss
}

\begin{document}

\maketitle

\begin{abstract}
We  define Poisson genericity  for  infinite sequences in 
any  countable alphabet  
with an invariant  
exponentially $\psi$-mixing probability measure.
A sequence is Poisson generic if the number of occurrences of blocks of symbols asymptotically follows a Poisson law as the block length increases.
We prove that almost all sequences are  Poisson generic. Our result generalizes Peres and Weiss' theorem about Poisson genericity of integer bases numeration systems. 
In particular, we obtain that the continued fraction expansions of almost all real numbers are Poisson generic.  
 \end{abstract}

\section{Introduction and statement of results}

Several years ago Zeev Rudnick defined  the notion of Poisson genericity for real numbers: 
a real number is Poisson generic for an integer base $b$ greater than or equal to $2$  if in its  base-$b$ expansion the number of occurrences of blocks of digits follows the Poisson distribution. 
Peres and Weiss gave a metric result showing that almost all real numbers, in the sense of Lebesgue measure, are Poisson generic, see~\cite{weiss2020,abm}.  A construction of a Poisson generic sequence for base $2$ appears in \cite{BH}.

In this note we  define Poisson genericity  for 
 infinite sequences in 
any  countable alphabet (finite or countably infinite)
with respect to an invariant probability measure that is exponentially-mixing.
A sequence is Poisson generic if the number of occurrences of blocks of symbols in the sequence converges to the Poisson law.
We prove that Poisson genericity holds with probability~$1$. 
Our initial motivation was to prove that almost all real numbers have Poisson generic 
continued fractions by extending the methods developed in~\cite{weiss2020,abm}. 
Theorem \ref{thm:main}
not only establishes this for continued fractions but also applies to any numeration system with 
invariant and exponentially $\psi$-mixing probability measure.

Let $\Omega $ be a countable alphabet with at least two symbols. For each positive integer~$k$, the set $\Omega^k$  consists of all words of length $k$ over the alphabet  $\Omega$.
 We write $\Omega^*$ for the set of all finite words and $\Omega^{\mathbb N}$ for the set of one-sided infinite words.

We write $w=w_1\cdots w_k$ for words in $\Omega^k$ and we use the letter $x$ for infinite words in $\Omega^\mathbb N$. 
For a word~$w$, $|w|$  is its length. 
We number the positions in words and  infinite sequences  starting from~$1$ 
and  we write $w[l,r]$ for the sub-sequence of~$w$  that begins in position~$l$ and ends in position~$r$. 
We use interval notation, with a square bracket when the set of integers includes 
the endpoint and a parenthesis to indicate that the endpoint is not included. The same convention is used for  $x\in \Omega^\mathbb N$.  
Given a $k \in \N$ and  a word $w$ of length $k$, the subset  $C(w)$  of $\Omega^{\mathbb N}$ defined by 
\[C_k(w)=\{x\in \Omega^{\mathbb N}: x[1,k]=w\}\]
is called the cylinder of $w$. 
All possible cylinders of any length generate a sigma-algebra $\mathcal B$. 
Finally, we assume that a measure $\mu$ is defined on the sigma-algebra $\mathcal B$ so that $(\Omega^{\mathbb N} ,\mathcal B,\mu )$ is a probability space.
 For every $k\in \mathbb N$, the projection of $\mu$ over the first $k$ coordinates is a measure on $\Omega^k$ that we denote by $\mu_k$. To shorten notation,  for a word $w\in \Omega^k$ we write 
\[
\mu_k(w)=\mu(C_k(w))
\]
and  for  $W \subseteq \Omega^k$ we  write
\[
\mu_k(W) = \sum_{w \in W} \mu_k(w).
\]
For  $j\in\N$, $x\in \Omega^{\N}$,  $k\in\N$ and $w \in \Omega^{k}$,   
we  write $I_j(x,w)$ for  the indicator function  that the word $w$ occurs in the sequence  $x$ starting at position~$j$,
\[
I_j(x,w)= {\mathbbm 1}_{ x[j, j+k) = w }.
\]
For each $k\in \N$, we define 
on the  space  $\Omega^{\mathbb N}\times\Omega^k$
with measure $\mu \times\mu_k$ 
the integer-valued  random measure 
$M_k=M_k(x,w)$ on  $\R^+=[0,+\infty)$ by
\[
M_k(x,w) (S)= \sum_{j :\, j  \mu_k(w) \in S} I_j(x, w),\quad \text{for any Borel set } S\subseteq \mathbb R^+ .
\]
We also define, for each $x \in \Omega^\N$ and $k \in \N$, the 
integer-valued random measure $M_k^x$ as 
\[
M_k^x(w)(S) = M_k(x,w)(S)
\quad \text{for any Borel set } S\subseteq \mathbb R^+. 
\]
Similarly, for each $k\in \mathbb N$ and each fixed $w\in \Omega^k$, we have the integer-random measure given by
\[
M_k^w(x)(S) = M_k(x,w)(S)
\quad \text{for any Borel set } S\subseteq \mathbb R^+.
\]

A  point process $Y(\cdot)$ on  $\R^+$ is an integer-valued random measure.
Therefore, $M_k(\cdot)$ and $M_k^{x}(\cdot)$ are point processes on  $\R^+$ for each $k\geq 1.$
The Poisson point process on~$\R^+$ 
is a point process $Y(\cdot)$ on $\R^+$  such that the following two conditions hold:
(a) 
for all disjoint Borel sets $S_1, \ldots, S_m$ included in $\R^+$, the random variables $Y(S_1),\ldots , Y(S_m)$ are mutually independent; and (b)  for each bounded Borel set $S\subseteq \R^+$, $Y(S)$  has the distribution of a  Poisson random variable with parameter equal to the Lebesgue measure of~$S$.
A sequence $\pare{Y_k(\cdot) }_{k\geq 1}$ of point processes converges in distribution to a point process $Y(\cdot)$  if, for every Borel set $S$, the random variables $Y_k(S)$ converge in distribution to $Y(S)$ as $k$ goes to infinity. 
A thorough presentation  on Poisson point processes can be read from~\cite{kingman} or~\cite{last}.

We write $p(\lambda,j)$ to denote the Poisson mass function  
 $e^{-\lambda} \lambda^j / j!$.

\begin{definition}[Poisson genericity]
We say that $x \in \Omega^\N$ is Poisson generic
if  the sequence 
$\pare{M_k^x(.)}_{k \in \N}$ 
of point processes on $\R^+$ 
converges in distribution to a Poisson point process on $\R^+$, as $k$ goes to infinity.
This means that  for  every Borel set $S\subseteq \R^+$, every integer $j\geq 0$,
\[
\mu_k(\{w\in \Omega^k: M_k^x(w)(S)=j\})\to p(|S|,j), \text{ as $k\to\infty$.}
\]
\end{definition}

We state Theorem~\ref{thm:main}, our main result, under the following assumptions on the probability measure $\mu$: 
\bigskip

\textbf{Exponentially $\psi$-mixing}: For each  $k  \in \mathbb N$  and $\ell\in \mathbb N$ such that $1\le k< \ell$ or $\ell=\infty$, consider the sigma-algebra $\mathcal B_{k,\ell}$ generated by the sets  
\[\{x\in \Omega^\mathbb N: I_k(x,w)=1\} \]
where $w$ is any word in $\Omega^{\ell-k}$ for finite $\ell$ or in $\Omega^\star$ for $\ell=\infty$. 
The measure $\mu$ is exponentially $\psi$-mixing if there are constants $\sigma \in (0,1)$ and $T>0$ so that for every  $k,\ell\in \N$  where  $ \ell >k $, for every $ A \in {\mathcal B}_{1,k}, B \in {\mathcal B}_{\ell,\infty}$ of non-zero measure,  
\begin{equation}
\label{eq:mixing}
\abs{\frac{\mu(A\cap B)}{\mu(A)\mu(B)}-1}\le T\sigma^{\ell-k} . 
\end{equation}

\noindent
In terms of words and indicator functions, this mixing property implies that  
    for every $ u,v\in\Omega^*$,
    for every  $i,j\in \N$  where  $ j >i +\abs{u}$,   
    \begin{equation} \label{eq:our-mixing}
      \left|  \frac{\mu(x\in \Omega^\N: I_i(x,u) I_j(x,v) = 1) }{ 
 \mu_{\abs{u}}(u) \mu_{\abs{v}}(v)} -1\right| \le T{\sigma^{j-i-|u|}}.
    \end{equation}

\textbf{Invariance}:
The measure $\mu$ is invariant if for every $k\in \N$, $w\in\Omega^k$ and $ i,j \in \N$, 
    \begin{equation}\label{eq:invariannt}
        \mu(x\in\Omega^\N : I_i(x,w) = 1) = \mu(x\in\Omega^\N : I_j(x,w) = 1).
    \end{equation}

\begin{theorem}[Main Result]
\label{thm:main}
For any invariant and exponentially $\psi$-mixing probability measure $\mu$  on $\Omega^{\mathbb N}$,
$\mu$-almost all $x \in \Omega^\N$ are Poisson generic.
\end{theorem}

Thus, Theorem \ref{thm:main} says that  for $\mu$-almost all $x\in \Omega^\mathbb N$,
for every Borel set $S\subseteq \R^+$,
for every integer $j\geq 0$,
\[
\mu_k\left(
\{w\in \Omega^k: M_k(x,w)(S)=j\}\right) \to p(|S|,j),\text{ as $k\to\infty$}.
\]
Theorem~\ref{thm:main} might be true with weaker assumptions on the mixing 
properties of the probability measure~$\mu$. 

Our result differs from the   metric results proved  in the  context of dynamical systems
which say that for each $t>0$,  there exists an exceptional set $E_t\subset \Omega^\mathbb N$ 
with $\mu(E_t)=0$ so that for any  $y\in \Omega^\mathbb N \setminus E_t$ and  every integer $j\ge 0$,
\[\mu\left(\left\{x\in \Omega^\mathbb N: M_k(x,y[1,k])((0,t))=j\right\}\right)\to p(t,j), \text{as $k\to \infty$.}  \]
 The key difference between the two statements lies in whether the point process 
draws randomly the first or the second argument of $M_k(x,w)$. 
Although the statements are symmetric, this symmetry does not extend to the proofs. 
The work on limit theorems for mixing sequences
dates back to Doeblin in 1940~\cite{Doeblin1940}
with his statement that 
the number of  occurrences of large digits in the partial quotients of 
continued fractions follows the Poisson law. This    was  later proved 
by Iosifescu~\cite{Iosifescu-Doeblin}, see also~\cite{Ghosh}.
These pioneering works on continued fractions have evolved into
 the study of the statistics of the number of visits of orbits under a 
given discrete dynamical system to a sequence of sets of positive 
measures shrinking to a point.
When the sequence of  visited sets is sufficiently regular and the 
dynamical system  
satisfies   mixing conditions, it can be shown that the number of visits    
follows a Poisson Law as the measures of the sets converge to zero.

Since the pioneering early works in the field, 
research has diversified significantly, 
including studies of families of dynamical systems under varying mixing conditions,
 statistics of periodic orbits, and analysis of error terms. 
For a historical overview and references, see \cite{HSV99} and \cite{Zwei}.


 Within dynamical systems, our Theorem \ref{thm:main} holds for fibred numerations  systems with an invariant probability measure satisfying the exponentially $\psi$-mixing condition \eqref{eq:mixing}. Following \cite[Definition 2.3]{BaBe06}, given a compact set $X$ and a map $T:X\to X$, we say that $(X,T)$ is a fibred system   if the 
transformation $T\colon X\to X$ is such that  there exist 
countable
set $\Omega$ and a topological  partition  ${ P}= \{C(a) \}_{{a\in {\Omega}}}$ of $X$ for which the
restriction $T_a$ of $T$ to $C(a)$ is injective and  continuous,  for each $a\in {\Omega}$. (Here topological partition means that sets $C(a)$ for $a\in \Omega$ are non-empty, open, and disjoint, and the union of their closures is the whole $X$.)
As it is proved in \cite{BaBe06}, a fibred system defines a numeration system called fibred numeration  system.

Fibred systems define by   piecewise monotonic maps of the interval with surjective branches satisfying the so-called AFU conditions,
 in the sense of \cite{Zwei}, and satisfying \cite[Theorem 1.b]{Aan} fulfill the assumptions of Theorem~\ref{thm:main}.
For instance, integer bases and their generalization to  numeration systems induced
 by 
countable set of independent  digits. Other examples are  the   classical, centered, 
and odd continued fractions algorithms. 
The map associated with continued fractions is the Gauss map 
\[T:[0,1]\to [0,1],\quad T(x)=1/x-\lfloor 1/x\rfloor, \text{ if $x\neq 0$  and $T(0)=0$}.\]
The Gauss measure defined as $d\mu(x)=dx/( (1+x)\ln 2)$ is invariant and   exponentially $\psi$-mixing for $T$, for a proof see, for instance, \cite[Theorem 1.2.1, Corollary 1.3.15]{IosiKraa}.

One example of piecewise monotonic maps (non-surjective) with  AFU conditions are beta shifts.  For a characterization of beta shifts having invariant and exponentially $\psi$-mixing  probabilities, see \cite[Section 2]{Aan}. 
There are also examples in two dimensions as the Ostrowski dynamical system given by the map $ T : [0, 1)^2 \to   [0, 1)^2,\ T (x, y) = (1/x-\lfloor 1/x\rfloor, y/x-\lfloor y/x\rfloor),\ T (0, y) = (0, 0)$. A detailed description of the corresponding fibred numeration system can be found in~\cite{BeLee23}. The map $T$ has an invariant and exponentially $\psi$-mixing probability measure, see \cite[Theorem 4.4]{BeLee23}.

Another point of view for numerations systems is to consider them as stochastic processes taking
 values in a 
countable alphabet such as irreducible and aperiodic Markov chains with 
a finite number of states.  Theorem \ref{thm:main}, of course, applies to these cases.

To prove Theorem~\ref{thm:main} we adapt  Peres and Weiss' proof strategy used 
in~\cite{weiss2020,abm}. 
The proof consists of two parts, an annealed result and a quenched result.
For the annealed result, for each word in $\Omega^k$ we
use the Chen--Stein method~\cite{Chen} to obtain a uniform bound for the rate of 
convergence of $M_k^w(.)$ to the Poisson law.
We then show that the sequence of point processes $M_k(.)$ on $\R^+$ converges in 
distribution to a Poisson point process on $\R^+$  in the product space $\Omega^\N\times \Omega^k$.
To obtain this part we prove  Lemma \ref{le:wfix}, which is a version for an infinite 
alphabet of   Abadi and Vergne's pointwise limit theorem with sharp error terms~\cite{AbadiVergne} 
which holds for a finite alphabet. 

The quenched result is an application of a concentration inequality that proves that 
what happens on average in the space $\Omega^\N$  is essentially what happens 
for almost all sequences. 
%
%
With this end in mind we prove a concentration inequality that holds for  alphabets and  for functions 
depending on countably many variables with 
the bounded differences property 
 known as Lipschitz condition for a weighted Hamming distance. For this we adapt 
the martingale difference method given in \cite[Theorem 3.3.1]{KontoThesis} and 
\cite[Theorem 1.1]{KontorovichRamanan2008}.  The former  deals with  a finite 
alphabet and weighted Hamming distances and  the latter deals with an infinite 
alphabet but a constant Hamming distance.  
 
To account for infinite memory numeration systems with digits in a countable 
alphabet we have   to work with functions depending on infinitely many variables.

\section{Proof of Theorem 
\ref{thm:main}}

The next proposition gives sufficient conditions for a sequence of random 
measures on $\R^+$  to converge to a Poisson random measure on  $\R^+$.  
These are conditions just on integer-valued random measures (point processes)  
on finite unions of disjoint intervals with rational endpoints.
It is later used in Lemmas~\ref{lemma:annealed}  and~\ref{prop:quenched}.

 \begin{proposition}[Instantiation of Kallenberg \protect{\cite[Theorem 4.18]{kallenberg2017random}}]  
\label{kallenberg}
Let $(X_k(\cdot))_{k\in\N}$ be a sequence of point processes on~$\R^+$ and
let~$Y(\cdot)$ be a 
Poisson process on~$\R^+$.
If for any $S\subseteq \R^+$ that is a  finite union  of disjoint intervals with rational
endpoints we have 
\begin{enumerate}
    \item $\limsup\limits_{k \to \infty}  \E[X_k(S)] \leq \E[Y(S)]$  and 
    \item $\lim\limits_{k \to \infty}  \prob\pare{  X_k(S) = 0 } = \prob\pare{  Y(S) = 0 }$
\end{enumerate}
then $X_k(\cdot)$ converges in distribution to $ Y(\cdot)$, as $k \to \infty$. 
\end{proposition}

\subsection{The annealed result}

Fix a real number  $\lambda>0$. For any $h:\mathbb N \to \mathbb R$, 
the number  $P(\lambda,h)$ denotes the average value of $h$ with 
respect to the Poisson distribution of parameter $\lambda$:  
\[
P(\lambda,h)=e^{-\lambda}\sum_{j=0}^\infty h(j)\lambda^j/j!.
\]
In particular, if $h(j)=\mathbbm 1_{j}$, we recover the   probability 
mass function $p(\lambda,j)$, where  $p(\lambda,j)=e^{-\lambda} \lambda^j/j!.$
When $W$ is a random variable defined on $\Omega^\mathbb N$, we write  
\[\E_\mu\corch{W}=\int_{\Omega^\mathbb N} W(x)d\mu(x).\]
If $W=\mathbbm 1_{S}$ for some $S\subset \Omega^{\mathbb N}$,
$\E_\mu\corch{\mathbbm 1_{S}}=\mu(S).$
Similarly, when $W$ is defined on $\Omega^\mathbb N\times \Omega^k$, we write
\[\E_{\mu\times \mu_k}\corch{W}=\int_{\Omega^\mathbb N\times \Omega^k}W(x,w)d\mu(x) d\mu_k(w).\]

Every exponentially $\psi$-mixing measure $\mu$ has a contraction ratio,~\cite[Lema 1]{Ab01}.

\begin{definition}[Contraction ratio $\rho$] Let $\mu$ be an invariant exponentially $\psi$-mixing  measure on $\Omega^\mathbb N$. 
The measure  $\mu$  has a contraction ratio $\rho\in (0,1)$ 
if there is constant $K > 0$ such that 
    for every $k\in \N$, for every $ w\in\Omega^k$,
\begin{equation}\label{eq:rhoK}
        \mu_k(w) \leq K {\rho}^k.
\end{equation}
\end{definition}

\begin{definition}[Bounded distortion constant $R$]
The mixing property for $\mu$ implies that there is a constant  $R>0$ such that, for every $u,v\in\Omega^*$, for every $i,\ell\ge 1$,
 \begin{equation} \label{bd}   \nonumber    \mu(x\in \Omega^\N: I_i(x,w) I_{i+\ell}(x,v) = 1) \le R
 \mu_{\abs{u}}(u) \mu_{\abs{v}}(v).
  \end{equation} 
We call this the bounded distortion property and we say that $R$ is the bounded distortion constant. 
\end{definition}

\begin{definition}[Periods of a word]\label{Def:period}
A word 
$w \in \Omega^k$ has period $\ell$ if $\ell < k$ and $w_i = w_{i+\ell}$ for all $1 \leq i \leq k-\ell$.
For  $w\in \Omega^k$, the set $\pi_w$ gathers the positive integers which are its periods, 
\[
\pi_w = \{\ell : w \text { has period } \ell \}.
\]
\end{definition}

\begin{definition}[Set $\J_{w, S}$]
For a given set $S\subset \mathbb R$ and a fixed word $w\in \Omega^k$, we define the set 
\[
\J_{w, S} = \set{i \in \N : i \mu_k(w) \in S}.
\]
\end{definition}

We use $\#A$ to denote the cardinality of a finite set  $A$.
We are interested in the cardinality of $\J_{w, S}$.
If $S$ is an interval $(a,b)$, for any $ a,b\in \mathbb R, a<b$,
\[ 
\frac{b-a}{\mu_k(w)} - 1\le \# \J_{w, S} \le \frac{b-a}{\mu_k(w)} + 1.
\]
Therefore, if $S$ is a finite union of $m$ nonempty intervals,
\begin{equation}
\label{eq: measure of S} \frac{\abs{S}}{\mu_k(w)} - m\le \#\J_{w,S}\le  \frac{\abs{S}}{\mu_k(w)} + m.
\end{equation}
That is, $\#\J_{w,S}=\abs{S}/{\mu_k(w)}+O(1)$ as $k\to \infty$ 
where the hidden constant of the $O$ term only depends on the number of intervals of  $S$.

\begin{lemma}\label{Le:exp} Let $\mu$ be an invariant  and exponentially $\psi$-mixing measure on $\Omega^{\mathbb N}$ with contraction ratio  $\rho$.
Let $S\subset \mathbb R^+$ be  a finite union of bounded intervals.
  For each $k\in \mathbb N$,  for each fixed $w \in \Omega^k$,  the following hold

\begin{enumerate}
\item 
$ \E_\mu\corch{M_k^w(S)}=
|S|+O(\rho^k),
 $
 \item $\E_\mu\corch{\pare{M_k^w(S)}^2} 
 =\abs{S} + \abs{S}^2  + O\pare{k \rho^k}  + O\pare{\sum_{\ell \in \pi_w} \rho^\ell}, $
 
 \item 
 $     \mathbb V_{\mu}(M_k^w(S)) = \abs{S} + O(k \rho^k) 
    + O\pare{\sum_{\ell \in \pi_w} \rho^\ell},
$

where $\pi_w$ are the periods of the word $w\in \Omega^k$ and the hidden constant in the $O$-term only depends on~$S$.
\end{enumerate}
\end{lemma}
\begin{proof}
We write $I_i^w(x)$ to denote $I_i(x,w)$.

{\em Point 1}. As a direct consequence of    \eqref{eq: measure of S} we have
\begin{align*}
 \E_\mu\corch{M_k^w(S)}&
 =\int_{\Omega^\mathbb N}\pare{M_k(x,w)(S)} d\mu(x)
 \\&=\sum_{i \in \J_{w,S}} \E_\mu\corch{ I_i^w } 
         \\&=\sum_{i \in \J_{w,S}} \mu_k(w) 
         \\&= \pare{\frac{\abs{S}}{\mu_k(w)} + O(1)} \mu_k(w). 
    \end{align*}
The contraction ratio property   
$\mu_k(w)=O(\rho^k)$  yields 
\[\E_\mu\corch{M_k^w(S)}= \abs{S} + O(\mu_k(w))  
     = \abs{S} + O(\rho^k).   
\]

{\em Point 2}. The random variable $\pare{M_k^w(S)}^2$ is a sum which involves the products $I_i^w I_j^w$, for $i,j\in \J_{w,  S}$. 
We split that sum  as follows:
 \[\pare{M_k^w(S)}^2= \sum_{i \in \J_{w,S}}     (I_i^w) ^ 2  +\sum_{i \in \J_{w,S}} 
            \sum_{\substack{j \in \J_{w,S} \\ 1\le \abs{i-j} < k}} I_i^w I_j^w +\sum_{i \in \J_{w,S}} 
            \sum_{\substack{j \in \J_{w,S} \\ \abs{i-j} \geq k}}  I_i^w I_j^w .\]
Let   
\begin{align*} 
    E_1 & = \sum_{i \in \J_{w,S}} \E_\mu  \corch{ (I_i^w) ^ 2} ;\\ 
    E_2 & = \sum_{i \in \J_{w,S}}
        \sum_{\substack{j \in \J_{w,S} \\ 1 \leq \abs{i-j} < k}} \E_\mu\corch{I_i^w I_j^w} 
        ;\\ 
    E_3 & = \sum_{i \in \J_{w,S}} 
        \sum_{\substack{j \in \J_{w,S} \\ \abs{i-j} \geq k}} \E_\mu\corch{I_i^w I_j^w} 
        .
\end{align*}
Notice that $\E_\mu\corch{(M^w_k(S))^2} =E_1+E_2+E_3$.

We prove that $E_1= \abs{S} + O(\rho^k)$.
This is a direct consequence of the fact that $(I_i^w)^2=I_i^w$ because $I_i^w$ is an indicator function and the estimate already proved for $\E_\mu\corch{M_k^w(S)}$: 
\begin{align*}
 E_1=   \sum_{i \in \J_{w,S}} \E_\mu\corch{\pare{I_i^w}^2} 
    &=\sum_{i \in \J_{w,S}} \E_\mu\corch{I_i^w} 
     =\E_\mu\corch{M_k^w(S)}=|S|+O(\rho^k).   
\end{align*}

 We prove that $E_2= 
        O\pare{\sum_{\ell \in \pi_w} 
        \rho^\ell}$.
The first step is based on the following observation: consider $i,j\in \mathbb N$, $j>i$. For a fixed $w$,
there exists $x\in \Omega^\N$ for which 
$I^w_iI^w_j(x)=1$ if and only if $j-i$ is a period of $w$, 
 $j-i\in \pi_w$,
\begin{align*}
  \sum_{i \in \J_{w,S}} 
    \sum_{\substack{j \in \J_{w,S} \\ 1 \leq \abs{i-j} < k}}
    \E_\mu\corch{I_i^w I_j^w} 
    &= 
2 \sum_{\ell \in \pi_w} 
    \sum_{\substack{i \in \J_{w,S} \\ i + \ell \in \J_{w,s}  }  }
    \E_\mu\corch{I_i^w I_{i + \ell}^w}. 
\end{align*}
The bounded distortion property, the invariance of the measure and, finally, the existence of a contraction ratio imply that
\begin{align*}
  E_2  &
 \le 
 2 \sum_{\ell \in \pi_w} 
    \sum_{\substack{i \in \J_{w,S} \\ i + \ell \in \J_{w,s}  }  }
    R \mu_k(w) \mu_{\ell}\pare{w[1 \ldots \ell]} 
    \leq 
     \sum_{\ell \in \pi_w} 
     2 R K \mu_k(w) \rho^\ell
    \sum_{\substack{i \in \J_{w,S}\\ i + \ell \in \J_{w,s}}} 1.
\end{align*}
Now, we use \eqref{eq: measure of S} in order to deal with $\#\J_{w,s}$ and obtain
\begin{align*}     
 E_2   \le
    \sum_{\ell \in \pi_w} 
    2 R K \mu_k(w) \rho^\ell  
        \pare{\frac{\abs{S}}{\mu_k(w)} + O(k)} 
    &= 
    O\pare{\sum_{\ell \in \pi_w}
    \rho^\ell + k \mu_k(w) \rho^{\ell}} 
  \\&=
    O\pare{\sum_{\ell \in \pi_w} \rho^\ell},
\end{align*}
which proves the estimate.
\smallskip

We prove that $E_3= \abs{S}^2 + O(k \rho^k)$.
First, by the mixing property \eqref{eq:our-mixing}and some manipulations:
\begin{align*}
   E_3= \sum_{i \in \J_{w,S}} 
    \sum_{\substack{j \in \J_{w,S} \\ \abs{i-j} \geq k}} 
    \E_\mu\corch{I_i^w I_j^w} 
  &= \sum_{i \in \J_{w,S}} 
    \sum_{\substack{j \in \J_{w,S} \\ \abs{i-j} \geq k}} 
    \mu_k^2(w) \pare{1 + O\pare{\sigma^{|i-j|-k}}}\\
    &=  
    \mu_k^2(w) \pare{
        \pare{
        \sum_{i \in \J_{w,S}} 
        \sum_{\substack{j \in \J_{w,S} \\ \abs{i-j} \geq k}} 1
        }
        + 
        O\pare{
        \sum_{i \in \J_{w,S}} 
        \sum_{\substack{j \in \J_{w,S} \\ \abs{i-j} \geq k}}
        \sigma^{|i-j|-k}.
        }
    }
    \end{align*}
Fix $i\in \J_{w,S}$. Notice that 
$\J_{w,S}=\pare{\J_{w,S}\cap \{j: |i-j|\ge k  \}} \cup \pare{\J_{w,S}\cap \{j: |i-j|< k  \}} $. 
The   set $\{j: |i-j|< k  \}$ has cardinality at most $2k$. Hence, with  \eqref{eq: measure of S},
\[\#\pare{\J_{w,S}\cap \{j: |i-j|\ge k  \}}=\# \J_{w,S} + O(k)=\frac{|S|}{\mu_k(w)}+O(k).\]
We move on to the next sum. Using the sum of the geometric series with $\sigma < 1$, 
we get the bound
\begin{align*}
\sum_{i \in \J_{w,s}} \sum_{\substack{j \in \J_{w,S} \\ \abs{i-j} \geq k}} 
\sigma^{|i-j| - k} \le \#{\J_{w,S}}\sum_{n = -\infty}^\infty \sigma^{|n|}  = O\pare{1 / \mu_k(w)}.
\end{align*}
The contraction ratio property yields   
    \begin{align*}
   E_3
    &= 
    \mu_k^2(w) \pare{ 
    \pare { \frac{\abs{S}}{\mu_k(w)} + O(1) }
    \pare { \frac{\abs{S}}{\mu_k(w)} + O(k)}
        + O\pare{1 / \mu_k(w)}
    }
    \\
    &= \abs{S}^2 + 
    O\pare{k \mu_k(w)} \\
    &= \abs{S}^2 + 
    O\pare{k \rho^k}.
\end{align*}
Since $\E_\mu[(M^w_k(S))^2] =E_1+E_2+E_3$, the above estimates complete the proof of this point.
\smallskip

{\em Point 3}. 
By definition,
    \begin{align*}
        \mathds{V}_{\mu}(M_k^w) &= \E_\mu\corch{(M_k^w)^2} - \E_\mu\corch{M_k^w} ^ 2 \\
        &= \abs{S} + \abs{S}^2 + O(k \rho^k)
        + 
        O\pare{\sum_{\ell \in \pi_w} \rho^\ell} 
        - \pare{\abs{S} + O(\rho^k)} ^ 2\\
               &= \abs{S} + O(k \rho^k) 
        + O\pare{\sum_{\ell \in \pi_w} \rho^\ell}.
    \end{align*}
\end{proof}

\begin{lemma}\label{Le:expprod} Let $\mu$ be an invariant exponentially $\psi$-mixing measure on $\Omega^{\N}$ with contraction ratio $\rho$. For each $k\in \N$, consider the projection $\mu_k$ of $\mu$ over $\Omega^k$. Let $S\subset \mathbb R^+$ be a finite union of bounded intervals. Then, the following holds
\[\E_{\mu\times \mu_k}[M_k(S)]=|S|+ O(\rho^k).
\]
\end{lemma}

\begin{proof}
By definition of $\E_{\mu\times \mu_k}[M_k(S)]$ and then Point 1 of   Lemma~\ref{Le:exp} we obtain,
\[\E_{\mu\times \mu_k}[M_k(S)]=\sum_{w\in \Omega^k}\mu_k(w)\E_\mu[M_k^w(S)]=\sum_{w\in \Omega^k}\mu_k(w)\left(|S|+O(\rho^k)\right)=|S|+O(\rho^k).\]
\end{proof}
The total variation distance $d_{TV}$ between two probability measures $Q$ and $R$ on a 
sigma-algebra~$\mathcal{F}$ is defined via
\begin{align*}
d_{TV}(Q,R)=\sup_{ A\in \mathcal{F}}\left|Q(A)-R(A)\right|.
\end{align*}

The total variation distance between two random variables~$X$ and $Y$ taking values in $\mathbb N$ is simply 
\begin{align*}
d_{TV}(X,Y) = \sup_{h:\mathbb N\to \mathbb R,\ |h|\le 1} \left|\E[h(X)] - \E[h(Y)]\right|.
\end{align*}
For each $w\in\Omega^*$ and each $S\subseteq R^+$ we bound the total variation distance between the distribution of the random variable $M^w_k(S)$ and the Poisson distribution 
using  Chen's result~\cite[Theorem 4.4]{Chen}, stated below as Proposition \ref{prop:chen}. It  considers a sequence (finite or infinite) of random variables $X_1, X_2, X_3\dots$ with the  following mixing condition:
\medskip

\textbf{Exponentially $\phi$-mixing condition}: 
 Let $i,j$ be natural numbers with $j>i$ and let $\mathcal B_{i,j}$ be the sigma-algebra generated by the random variables $X_i,...,X_j$.  With $\phi:\mathbb N \to \mathbb R$, $\phi(m)=e^{-\alpha m}$ for some $\alpha>0$,   for every 
 for every $A \in  \mathcal{B}_{1,i} $ and for every 
 $B\in \mathcal{B}_{j, \infty} $, 
\begin{equation}\label{eq:ChM}
\abs{
\frac{{\mathbb P}(A\cap B)}{{\mathbb P}(A)}-\mathbb P(B)}\le \phi(j-i)\quad \text{ for any}\ j>i. 
\end{equation}
 Notice that our  exponentially $\psi$-mixing condition \eqref{eq:our-mixing}
implies  exponentially $\phi$-mixing \eqref{eq:ChM}.

\begin{proposition}[{{\protect\cite[Chen's Theorem 4.4]{Chen}}}] 
\label{prop:chen}
Let $X_1,\dots, X_n$ be a sequence of identically distributed random variables taking values in $\{0,1\}$ and satisfying the exponentially $\phi$-mixing condition \eqref{eq:ChM} for some $\alpha>0$.   
Define \[W=\sum_{i=1}^n X_i \quad \text{ and } \quad \lambda=\E\corch{W}. \]
Then, for any $h:\mathbb N\to \mathbb R$ with $|h|\leq 1$ and $n\geq 3$,  
\[
|\E(h(W))-P(\lambda,h)|< C_1(\alpha) \min(\lambda^{-1/2},1) 
\Big(\mathbb V(W) -\lambda+ (\lambda+1)^2 \frac{\log n}{n}\Big)
\]
where $\E\corch{W}$ and $\mathbb V[W]$  denote the expectation  and variation of $W$ respectively, and    $C_1(\alpha)$ 
depends only on $\alpha$. 
\end{proposition}


\begin{proposition}[\protect{\cite[Lemma 5.2]{Chen}}] \label{ChenB}
For any real positive $\lambda$ and $t$ and for any $h:\mathbb N\to \mathbb R$ with $|h|\leq 1$, 
\[\abs{P(\lambda,h)-P(t,h)}\le 2|\lambda -t|.\]
 \end{proposition}

\begin{lemma}\label{L:n} Fix $S\subseteq \mathbb R^+$  a finite union of bounded intervals. For each $k\in \mathbb N$ and $w \in \Omega^k$ the quantity     $n(k,w)=\#\J_w(S)$  satisfies 
\[\frac{\log(n(k,w))}{n(k,w)}=O(k\rho^k).\]
\end{lemma}
\begin{proof}[Proof of Lemma~\ref{L:n}]  We use identity \eqref{eq: measure of S}. By the existence of a contraction ratio, there is a constant $K>0$, for which
\[n(k,w)=\#\J_w(S)\ge \frac{|S|}{K\rho^k}- m \]
where $m$ is the number of intervals of $S$. We are assuming that $m$ does not depend on $k$. Since $\rho^k\to 0$ as $k\to \infty$, for $k$ large enough, 
\[\frac{\abs{S}}{K\rho^k}-m\ge \frac{\abs{S}}{2K\rho^k},\]
and $\log n/n$ is decreasing as $n\to \infty$, 
\[\frac{\log(n(k,w))}{  n(k,w)}\le \frac{ \log\pare{{\abs{S}}/({2K \rho^k})}}{\abs{S}/(2K \rho^k)} =O(k \rho^k).
\]
\end{proof}

The following result extends the one in~\cite{AbadiVergne} to an alphabet with infinitely many symbols.

\begin{lemma}[ {{Total variation distance between
$\E_{\mu}[h(M^w_k(S))]$ and $P({\abs{S}}, h)$ }}]\label{le:wfix} 
Let $\mu$ be an invariant and exponentially $\psi$-mixing  measure
and let  $S\subset \mathbb R $ be a finite union of bounded intervals.  There exists  $k_0\in \mathbb N$ and $C>0$ such that for every $k\ge k_0$ and for every $w \in \Omega^k$,  for any $h:\mathbb N \to \mathbb R$ such that $|h|\le 1$, 
\[\abs{\E_{\mu}[h(M^w_k(S))]
- P({\abs{S}} ,h)} \le C\pare{ k \rho^k+\sum_{\ell\in \pi_w} \rho^\ell } .\]
The constant $C$ depends on the measure $|S|$ and the number of intervals of the set $S$.     
\end{lemma}

\begin{proof}[Proof of Lemma~\ref{le:wfix}] 
Since \eqref{eq:our-mixing}
implies  \eqref{eq:ChM}, we apply Proposition~\ref{prop:chen}.
Fix the set $S\subset \mathbb R$. For each  $k\in \mathbb N$ and $w\in \Omega^k$, we consider   the  random variables
$M_k^w(S)$, $k\ge 1$.  
Since we assumed \eqref{eq:invariannt} the measure $\mu$ is invariant, so 
 the variables $I_i(x,w)$ are identically distributed.
Consider
\[
M_k^w(S) = \sum_{i \in \J_{w,S}}   
I_i(x,w), \quad
\lambda=\lambda(k,w) =\E_\mu\corch{M_k^w(S)}  
\text{ and } n= n(k,w) = \#\J_{w,S}.
\]
By Lemma~\ref{Le:exp},
$
\E_\mu[M^w_k]=\abs{S}+O(\rho^k)$
and
\[\abs{\mathbb V_{\mu}[M^w_k]-\E_\mu[M^w_k]}
=O\pare{ k \rho^k+\sum_{\ell \in \pi_w} \rho^\ell }.
\]
Since the variables $I_i(x,w)$ are 
exponentially $\psi$-mixing, 
for $k$ large enough, 
$n(k,w)\ge 3$. 
Using the bound on $\log(n(k,w))/n(k,w)$  in Lemma~\ref{L:n},
\begin{align*} 
    \abs{\E_\mu\corch{h(M_k^w)} - P(\lambda,h))}
    & \leq C_1(\alpha) \pare{\mathbb V[M_k^w] - \lambda + (\lambda+1)^2 \frac{\log n}{n}} 
  \\& =  
   O\pare{k \rho^k+\sum_{\ell \in \pi_w} 
   \rho^\ell }. 
\end{align*}
And by Proposition~\ref{ChenB},
\begin{align*}
    \abs{\E_\mu[h(M^w_k)] - P(\abs{S},h)}
    &\leq \abs{\E_\mu\corch{h(M_k^w)} - P(\lambda,h))}+\abs{P(\lambda,h)-P(\abs{S},h)} 
\\
  &  =
    O\pare{  k\rho^k+\sum_{\ell \in \pi_w} \rho^\ell}.
\end{align*}  
\end{proof}

\begin{lemma} \label{period-bound} Let $\ell$ and $k$ be two natural numbers so that $1\le \ell<k$. Let $W_k^\ell$ be the set of words $w\in \Omega^k$ with period $\ell$. 
 Then,
$$ 
\mu_k(W_k^\ell) = O(\rho^{k-\ell}).
$$
\end{lemma}

\begin{proof}[Proof of Lemma~\ref{period-bound}]
For a word $v \in \Omega^\ell$ and an integer $k \geq 0$, define $w=\operatorname{ext}(v, k)$  as the unique word $w$ in $\Omega^k$ such that 
\[w_i = v_{j}\quad \text{ if }\quad  i \equiv j \mod \ell\]
for any $i,j\in \mathbb N$, $1\le i \le k$ and $1\le j\le \ell$.

With the bounded distortion property 
and using the contraction ratio, we have, for any $w\in W_k^\ell$ there exists $v \in \Omega^\ell$ such that $w = v \operatorname{ext}(v,k-\ell)$ and therefore
\[
\mu_k(w) = \mu_k(v\; \operatorname{ext}(v,k-\ell))\le R \mu_{\ell}(v)\mu_{k-\ell}(\operatorname{ext}(v,k-\ell))\le RK \mu_{\ell}(v)\rho^{k-\ell}.
\]
In order to bound $\mu_k(W_k^\ell)$, observe that there is a bijective function $f: W_k^\ell \to \Omega^\ell$ which is simply defined as $f(w) = w[1 \ldots \ell]$.
Then,
\begin{align*}
\mu_k(W_k^\ell) 
    = \sum_{w \in W_k^\ell} \mu_k(w)
    = \sum_{v \in \Omega^\ell} \mu_k\pare{v \; \operatorname{ext}(v,k-\ell)}\le RK \rho^{k-\ell}\sum_{v \in \Omega^\ell}\mu_\ell(v)=RK\rho^{k-\ell}.
\end{align*}
\end{proof}

\begin{lemma}[$\mu\times\mu_k$-expectation on $M_k(S) $ in $\Omega^\N\times\Omega^k$]\label{le:poisfi}
Let $\mu$ be an invariant and exponentially $\psi$-mixing measure on $\Omega^\mathbb N$ with a contraction ratio $ \rho\in (0,1) $. For any set $S\subset \mathbb R$ which is a finite union of bounded intervals and for any $h:\mathbb N \to \mathbb R$ such that $|h|\le 1$,  the sequence of random variables  $\left(M_k(S)\right)_{k\ge 1}$ satisfies that
\[
\abs{\E_{\mu\times\mu_k}[h(M_k(S))]-P(|S|,h)}= O(k\rho^k).
\]
The constant hidden in the $O$-term depends on the measure $|S|$ and the number of intervals of the set $S$. 
\end{lemma}
\begin{proof}[Proof of Lemma~\ref{le:poisfi}]

We deal with the product measure $\mu\times \mu_k$ on $\Omega\times \Omega_k$. By definition, 
\[
{\E_{\mu\times\mu_k}[h(M_k(S))](h)}=\sum_{w\in \Omega^k}\mu_k(w)\E_{\mu}[h(M_k^w(S))].
\]
So, writing 
$P(|S|,h)=\sum_{w\in \Omega^k}\mu(w)P(|S|,h)$ 
and using the triangular inequality, we have
\[
\abs{\E_{\mu\times\mu_k}[h(M_k(S))]-
P(|S|,h)}\le {\sum_{w\in \Omega^k}\mu_k(w)\abs{\E_\mu[h(M_k^w(S))]-
P(|S|,h)}}.\]
The bounds on $\abs{\E[h(M^w(S))]-P(|S|,h)}$ given in Lemma~\ref{le:wfix} yield
\[
\abs{\E_{\mu\times\mu_k}[h(M_k(S))]-P(|S|,h)}= O(k\rho^k) + O\pare{\sum_{w\in \Omega^k}\mu_k(w) 
\sum_{\ell\in \pi_w}  \rho^\ell}.
\]
Now, the following holds, with the help of Lemma~\ref{period-bound}, 
\[\sum_{w\in \Omega^k}\mu_k(w) 
\sum_{\ell\in \pi_w} 
\rho^\ell= \sum_{\ell=1}^{k-1} \sum_{w \in W_k^\ell} \mu_k(w) \rho^\ell=\sum_{\ell=1}^{k-1} \rho^\ell \mu_k\pare{W_k^\ell}\le O\pare{\sum_{\ell=1}^{k-1} \rho^\ell \rho^{k - \ell} } 
     = O(k \rho^k).
 \]
Finally,
\[\abs{\E_{\mu\times\mu_k}[h(M_k(S))]-P({|S|},h)}= O(k\rho^k).\]
\end{proof}

\begin{lemma}
[The annealed result]
\label{lemma:annealed}
Let $\mu$ be an invariant and exponentially $\psi$-mixing measure on $\Omega^{\mathbb N}$. The sequence of random measures $(M_k(.))_{k \geq 1}$ on $\R^+$ converges in distribution to a Poisson point process on $\R^+$ as $k$ goes to infinity. More precisely, for any Borel set 
 $S\subseteq \R^+$,
 \[\mu\times \mu_k\pare{(x,w)\in \Omega^\mathbb N \times \Omega^k: M_k(x,w)(S)=j}\to
 p(|S|,j)
 \text{ as}\ k\to \infty.\]
 \end{lemma}

\begin{proof}[Proof of Lemma~\ref{lemma:annealed}]
 To prove it we apply Kallenberg's result stated in Proposition~\ref{kallenberg} for the product measure $\mu\times \mu_k$ on $\Omega\times \Omega_k$.
 In order to prove the first condition in Kallenberg's Theorem, we remark that 
 the expectation of the Poisson distribution of parameter $\abs{S}$ is $|S|$. On the other hand, the expectation of $M_k(S)$ is displayed in Lemma~\ref{Le:expprod}:  $\E_{\mu\times\mu_k}\corch{M_k(S)} = |S| + O(\rho^k)$, which implies that $$
 \limsup_{k \to \infty} \E_{\mu\times \mu_k}\corch{M_k(S)} \leq  |S|.
 $$
 Now, we prove the second condition of Proposition~\ref{kallenberg}. 
 Lemma~\ref{le:poisfi}, specified on   $h=\mathbbm{1}_{\{0\}}$,   implies that 
 $$
 \lim_{k \to \infty} \E_{\mu\times \mu_k}\corch{\mathbbm{1}_{\{0\}}(M_k(S))}  =p(\abs{S},0)=e^{-|S|}
 $$
 Then, both conditions of  Proposition~\ref{kallenberg} hold for every finite union of intervals $S$ and we can conclude that $\pare{M_k(.)}_{k \geq 1}$ converges in distribution to a  Poisson point process on $\R^+$.
 \end{proof}

\subsection{The quenched result}

Following a similar strategy as the one from~\cite[Proposition 3]{abm},
 we prove a concentration result to show that $M_k^x$ converges to Poisson for $\mu$-almost all $x\in\Omega^\N$.
We need an 
inequality that holds for countable alphabets  and functions depending on countably many variables with 
the bounded differences property 
 known as the Lipschitz condition for a weighted Hamming distance, see Proposition \ref{lemma:KontRam}. As mentioned in the Introduction, for this we adapt the martingale difference method given in \cite[Theorem 3.3.1]{KontoThesis} and \cite[Theorem 1.1]{KontorovichRamanan2008}. The former deals with a finite alphabet and weighted Hamming distances and the latter deals with
 a countable alphabet but a constant Hamming distance.

    We need to deal with infinitely many random variables $X_1, X_2, \ldots $ 
    defined on a space $\mathbb{X }$ taking values in the countable  alphabet $\Omega,$
    and a function 
    $\varphi: \Omega^\N \to \R$. 
    We  work with finitely many $X_1,\ldots,X_N$ and 
    $\varphi_N: \Omega^N \to \R$ 
    for $N$ large and then take $N \to \infty.$
    For each $N$, consider the filtration         $$
    {\cal F}_N = \sigma(X_1,\ldots, X_N) 
    $$
    of sigma-algebras generated by the first $N$ random variables and the Azuma--Hoeffding coefficients that are defined in terms of martingale differences. 

\begin{definition}[Azuma--Hoeffding  coefficients $d_i$]
Let $ X_1, {X_2}, X_3, \dots$  be a sequence of random variables defined on $(\mathbb X, {\cal F}, \prob{ })$
taking values on some $\Omega.$ Consider the filtration of sub-sigma-algebras~${\cal F}_i$, \begin{align*}
\set{\emptyset, \mathbb X} = {\cal F}_0\subseteq   
{\cal F}_1 \subseteq \ldots \subseteq  {\cal F}.
\end{align*}
and $\mathcal F$ is the smallest sigma-algebra containing all the others.
Consider a function  $\varphi :\Omega^{\N} \to \R$. For  $X=X_1 X_2\cdots$, we define, for each $i\in \mathbb N $, 
\begin{align*}
    V_i(\varphi) = \E\corch{\varphi(X) \, | \, {\cal F}_i } - 
    \E\corch{\varphi(X) \, | \, {\cal F}_{i-1} }
\end{align*}
and we  define  the Azuma--Hoeffding  coefficients 
\begin{align}\label{d:AzuHoef}
d_i = 
\sup
_{\mathbb X}  |V_i(\varphi)|.    
\end{align}
\end{definition}

We follow~\cite[page 23]{KontoThesis} and we the use $\eta$-mixing coefficients introduced there.

\begin{definition}[$\eta$-mixing coefficients  and their associated matrix $\Delta$]
For a sequence $X_1, X_2, \ldots $ of random variables taking values
in a
countable alphabet $\Omega$ and for any $i,j\in \mathbb N$, we write
$X^{\geq j}$ and $X^{\leq i}$ for $(X_j, X_{j+1}, \ldots )$
and $(X_1, \ldots, X_i)$ respectively. 
 The mixing coefficients $\eta_{i,j}$ are real  numbers such that, for $j\ge i$,
\begin{align}\nonumber
    \eta_{ij} =\sup_{
    \begin{subarray}{c}
    x,x'\in \Omega^i,  
    \\ x[i]\neq x'[i]
    \end{subarray}
    }
    \sup_{A\in \sigma(X^{\geq j} ) } 
    \abs{ 
    \prob\pare{ X^{\geq j }\in A | X^{\leq i} = x} -
    \prob\pare{ X^{\geq j }\in A | X^{\leq i} = x'} 
    } 
\end{align}
that is, the  supremum is taken over  $x$ and $x'$ which are elements of $\Omega^i$ which differ only in the $i$\ts{th} coordinate and 
$\sigma(X^{\geq j} )$
is the sigma-algebra of Borel sets not depending
the first $j-1$ coordinates.
If $\prob\corch{{X^{\le i}=x}}=0$ or  $\prob\corch{X^{\le i}=x'}=0$, we define $\eta_{ij}=0$. 
The corresponding  matrix $\Delta$  is defined as 
\begin{align*}\label{Deltaij}
    \Delta_{ij} = \begin{cases}
        \eta _{ij}, & \textit{ if } j > i  \\
        1, & \textit{ if } j = i  \\        
        0, & \textit{ if } j < i.   
    \end{cases}
\end{align*}
 \end{definition}

  We consider  the Banach space $(\ell^2,\|\phantom{v}\|)$ which is formed by all the real sequences \[ {v}=\left(v_j\right)_{j\in \mathbb N}\quad \text{so that}\quad \| {v}\|^2=\sum_{j\ge 1}
v_j^2<\infty.\] 
The matrix $\Delta$   induces a bounded linear operator $\Delta:\ell^2\to \ell^2$
 whose norm is defined as 
\[\|\Delta\| =\sup_{\| {v}\|=1} \|\Delta  v\|.
\]

The functions we are going to consider in the concentration inequalities
below satisfy a bounded differences property 
 known as Lipschitz condition for a weighted Hamming distance.

 \begin{definition}[$c$-Lipschitz] Given $c = (c_i)_{i\geq 1}$, with $c_i\in \mathbb R$, we say that the function $\varphi : \Omega^ \N \to \R$ 
is $c$-Lipschitz 
if for any
 $x, x'\in \Omega^ \N$  which differ only in coordinate $i$, we have, 
\begin{align}
\nonumber
    \abs{\varphi (x) - \varphi (x')} \leq c_i.
\end{align}
\end{definition}

\begin{proposition}[Upper bound for the Azuma--Hoeffding coefficients]
\label{l:dotprodbound}
Let $X_1,X_2,\dots $ a sequence of random variables defined on $(\mathbb X, \mathcal{F},\mathbb P)$ taking values on alphabet $\Omega$. Let   $\Delta$ be the matrix of $\eta$-mixing coefficients with $\|\Delta\|<\infty$.  Let $c=(c_i)_{i\ge 1}\in \ell^2$ and $\varphi:\Omega^\N \to \mathbb R$ 
a bounded function that is also $c$-Lipschitz. 
Then, the Azuma--Hoeffding coefficients $d=\left(d_i\right)_{i\ge 1}$ satisfy, for each $i\geq 1 $, 
\begin{align}\label{e:convdtoc}
   0\le d_i \leq \sum_{j\geq i} c_j \eta_{ij}.
\end{align}
In particular, $\|d\|\le \|\Delta\|\|c\|$. 
\end{proposition}

\begin{proof}
For each $i$, consider the filtration        
    ${\cal F}_i = \sigma(X_1,\ldots, X_i) 
    $
    of sigma-algebras generated by the first $i$ random variables.
     We want to  bound the Azuma--Hoeffding coefficients  
     as in~\eqref{d:AzuHoef}. 
   We  prove~\eqref{e:convdtoc} for $i=1$, the others are similar. 
For $x\in \Omega^N$ we denote 
$\E\corch{\varphi|X=x}$ by $\E\corch{\varphi|x}$ and similarly 
for probabilities. 
Let us fix $x\in \Omega^{N} $ so that $\mathbb P[x^{\le 1}]>0$.
Then,

\setlength{\mathindent}{0pt}

\noindent
\begin{align*}
V_{1}(\varphi)= & 
 \E\corch{\varphi | x^{\leq 1} } - 
 \E\corch{\varphi  }    \\  
 = &
 \int_{y\in \Omega^{\N} } 
  \E\corch{\varphi | x^{\leq 1} y^{> 1} } 
  {\rm d} \prob\pare{ x^{\leq 1} y^{> 1} | x^{\leq 1} }
  - 
 \E\corch{\varphi | y }  {\rm d}\prob\pare{  y  }
\\
= &
 \int_{y\in \Omega^{\N} } 
  \pare{ 
  \E\corch{\varphi | x^{\leq 1} y^{> 1} } -
  \E\corch{\varphi | y }
  } {\rm d}\prob\pare{  y  }
  +
 \int_{y\in \Omega^{\N} } 
  \E\corch{\varphi | x^{\leq 1} y^{> 1} } 
  \pare{
  {\rm d} \prob\pare{ x^{\leq 1} y^{> 1} | x^{\leq 1} }
  -  
 { \rm d} \prob\pare{ y }
  }.
\end{align*}

\setlength{\mathindent}{0pt}


Telescoping the last integrand for $N\in \N$ and $N\ge 2$, we get
\begin{align*}
V_{1}(\varphi) = &
 \int_{y\in \Omega^{\N} } 
  \pare{ 
  \E\corch{\varphi | x^{\leq 1} y^{> 1} } -
  \E\corch{\varphi | y^{\leq 1} y^{> 1} }
  } {\rm d}\prob\pare{  y  } \\ 
  &  +
 \int_{y\in \Omega^{\N} }
 \pare{
 \sum_{k = 2}^{N}
  \E\corch{\varphi | x^{< k} y^{\geq k} } 
  -
  \E\corch{\varphi | x^{\leq k} y^{> k} } 
  }
  \pare{
  {\rm d} \prob\pare{ x^{\leq 1} y^{> 1} | x^{\leq 1} }
  -  
  {\rm d} \prob\pare{ y }
  } \\ 
  &  +
 \int_{y\in \Omega^{\N} }
  \E\corch{\varphi | x^{\leq N} y^{> N} } 
  \pare{
  {\rm d} \prob\pare{ x^{\leq 1} y^{> 1} | x^{\leq 1} }
  -  
  {\rm d} \prob\pare{ y }
  }.
\end{align*}
\setlength{\mathindent}{0pt}

\noindent
Interchanging sum with integral in the second term leads to
\begin{align*}
V_{1}(\varphi) = &
 \int_{y\in \Omega^{\N} } 
  \pare{ 
  \E\corch{\varphi | x^{\leq 1} y^{> 1} } -
  \E\corch{\varphi | y^{\leq 1} y^{> 1} }
  } {\rm d}\prob\pare{  y  } \\ 
  &  +
 \sum_{k = 2}^{N}
 \int_{y\in \Omega^{\N} }
 \pare{
  \E\corch{\varphi | x^{< k} y^{\geq k} } 
  -
  \E\corch{\varphi | x^{\leq k} y^{> k} } 
  }
  \pare{
  {\rm d} \prob\pare{ x^{\leq 1} y^{> 1} | x^{\leq 1} }
  -  
  {\rm d} \prob\pare{ y }
  } \\ 
  &  +
 \int_{y\in \Omega^{\N} }
  \E\corch{\varphi | x^{\leq N} y^{> N} } 
  \pare{
  {\rm d} \prob\pare{ x^{\leq 1} y^{> 1} | x^{\leq 1} }
  -  
  {\rm d} \prob\pare{ y }
  }.
\end{align*}
Since both
$\E\corch{\varphi | x^{< k} y^{\geq k} }$ 
and 
$\E\corch{\varphi | x^{\leq k} y^{> k} }$
are measurable for the $\sigma-$algebra 
$\sigma(x^{\leq 1}X^{\geq k}| x^{\leq 1}) $
and
$\sigma(x^{\leq 1}X^{>1}| x^{\leq 1})$ is finer,
by the stability of conditional expectations 
we can replace ${\rm d}\prob\pare{ x^{\leq 1} y^{> 1} | x^{\leq 1} }$ by ${\rm d}\prob\pare{ x^{\leq 1} y^{\geq k} | x^{\leq 1} } $ and     $
  {\rm d} \prob\pare{ y }
  $ 
  by
  ${\rm d} \prob\pare{ y^{\geq k}} 
  $ for each $ 2\le k\le N$ in the middle summation, and we can replace ${\rm d} \prob\pare{ x^{\leq 1} y^{> 1} | x^{\leq 1} }$ by ${\rm d} \prob\pare{ x^{\leq 1} y^{> N} | x^{\leq 1} }$ in the last integral. 

\begin{align*}
V_{1}(\varphi)  & = 
 \int_{y\in \Omega^{\N} } 
  \pare{ 
  \E\corch{\varphi | x^{\leq 1} y^{> 1} } -
  \E\corch{\varphi | y^{\leq 1} y^{> 1} }
  } {\rm d}\prob\pare{  y  } \\ 
  &  +
 {\sum_{k=2}^N\int_{ y \in \Omega^{\N} } 
   \pare{    \E\corch{  
     \varphi | { x^{ < k } y^{\geq k} }} - 
     \E\corch{\varphi| { x^{ \leq k } y^{ > k}} 
     }}
\pare{  {\rm d} \prob\pare{ x^{\leq 1} y^{\geq k} | x^{\leq 1} } 
      - {\rm d} \prob\pare{ y^{\geq k} } }}
   \\
  &  +
 \int_{y\in \Omega^{\N} }
  \E\corch{\varphi | x^{\leq N} y^{> N} } 
  \pare{
  {\rm d} \prob\pare{ x^{\leq 1} y^{> N} | x^{\leq 1} }
  -  
  {\rm d} \prob\pare{ y }  }
\end{align*}
from which
\begin{align*}
V_{1}(\varphi)
  &\leq \sum_{j=1}^{ N} c_j \eta_{1j} + \sup_{\Omega^\N}{| \varphi|}  \cdot \eta_{1,N+1}.
\end{align*}

Since $||\Delta||<\infty$, $\eta_{1,N+1}$ goes to zero when $N\to \infty$, and 
$d_1\le\sum_{j\geq 1} c_j \eta_{1j} .
$
The proof is then complete.
\end{proof}

\begin{proposition}[Instantiation of \protect{{\cite[Lemma 4.1]{Ledoux}}}] \label{p:hoeffazuma}
Let $X =X_1, {X_2},  \dots, X_N$  be a sequence of random variables defined on $(\mathbb X, {\cal F}, \prob{ })$
taking values on some $\Omega.$ 
For $1\le i \le N$, consider the Azuma--Hoeffding coefficients $d_i$
as in~\eqref{d:AzuHoef}.
 Then,   
\begin{align*}\displaystyle
    \prob\pare{ \abs{\varphi(X) - \E [\varphi(X)] } \geq t } 
    \leq 
    2\exp\left( \frac{-t^2}{ 2 \sum_{i=1}^N d_i^2} \right).
\end{align*}
\end{proposition}

\begin{proposition}
[Concentration inequality with infinitely many variables]\label{lemma:KontRam}
Let $X_1, X_2,\dots $ be 
random variables taking values in some countable set $\Omega$ and  
 let   $\Delta$ be  the matrix of mixing coefficients with $\|\Delta\|<\infty$. Let  $\varphi:\Omega^\N \to \R$ be  a $c$-Lipschitz function for some $c=(c_i)_{i\ge 1}\in \ell^2$ such that $ \varphi(X) \in L^1$. 
Then, for any $t > 0$,
\begin{align*}
    \prob\pare{ \abs{\varphi(X) - \E [\varphi(X)] } \geq t } 
    \leq 
    2\exp\pare{ \frac{-t^2}{2{\|\Delta\| }^2 \|c\|^2 }}.
\end{align*}
\end{proposition}

\begin{proof} 
We assume first that $\varphi$ is bounded.
For each $N\in\N$ consider $\varphi:\Omega^\N\to\R$,
    \[
\varphi_N(x)=\E\corch{\varphi (X) \,| \,X^{\leq N} = x^{\leq N}}.
    \]
In other words, $\varphi_N$  is the martingale 
$\E\corch{\varphi|{\cal F}_N }.$
Since $\varphi$ is bounded, 
by Lévy's zero-one law 
the martingale $\E\corch{\varphi|{\cal F}_N}$ converges to 
$\varphi(X)$ 
both a.e. and in $L^1.$    
By standard properties of conditional expectations, it is easy to see that 
\[
V_i(\varphi_N) = 
 \begin{cases}
V_i(\varphi)  & \text{ if } i\leq  N \\ 
0 & \text{ if } i > N. 
 \end{cases}
\]
Hence, the Azuma--Hoeffding coefficients $d_i$ for $\varphi$ agree with those for
$\varphi_N$ as long as $i\leq N$ and vanish afterwards.

We  invoke Proposition~\ref{p:hoeffazuma} with $\varphi_N$ 
and get
\begin{align*}\displaystyle
    \prob\pare{ \abs{\varphi_N(X) - \E [\varphi_N(X)] } \geq t } 
    \leq 
    2\exp\left( \frac{-t^2}{ 2 \sum_{i=1}^N d_i^2} \right).
\end{align*}
By Proposition~\ref{l:dotprodbound} (keep in mind we are assuming $\varphi$ bounded now) this becomes
\begin{align*}\displaystyle
    \prob\pare{ \abs{\varphi_N(X) - \E [\varphi_N(X)] } \geq t } 
    \leq 
    2\exp\pare{ \frac{-t^2}{2\norm{{\Delta }}^2 \norm{c}^2 }}.
\end{align*}
The proof  for the case $\varphi$ being bounded finishes by taking the limit as $N\to \infty.$

Now  remove the boundedness hypothesis on $\varphi$.
By dominated convergence, the sequence of bounded functions
\begin{align*}
    \varphi_N = 
 \begin{cases}
N, & \text{ if } \varphi(x) > N \\ 
\varphi(x), & \text{ if } \abs{\varphi(x)} \leq N \\ 
-N, & \text{ if } \varphi(x) < -N 
 \end{cases}
\end{align*}
converges to  
$ \varphi $
both a.e. and in $L^1.$
It remains to check that these $\varphi_N$'s are still $c-$Lipschitz.
We can think of the truncations as $\varphi_N = \Psi_N \circ \varphi$,
with
\begin{align*}
 \Psi_N (x) = \frac{\abs{x+N} - \abs{x-N}}{2} = 
 \begin{cases}
N, & \text{ if } x > N \\ 
x, & \text{ if } \abs{x} \leq N \\ 
-N, & \text{ if } x < -N 
 \end{cases}
\end{align*}
having Lipschitz constant $1$.
Therefore, for every pair $x, x'\in \Omega ^\N$ differing only in the $i$th coordinate we have
\[
\abs{\varphi_N(x) - \varphi_N(x') }  
= \abs{\Psi_N(\varphi(x)) - \Psi_N(\varphi(x')) } 
\leq 
\abs{\varphi(x) - \varphi(x') } \leq c_i.
\]
This completes the proof.
\end{proof}

The next results consider our particular case of 
$\mathbb{X} = \Omega^\N $, the invariant exponentially $\psi$-mixing  measure $\mu$, the usual projections $X_i:\Omega^\mathbb N\to \Omega$ 
onto the $i$-th coordinate of $x\in \Omega^\N$
and ${\mathcal F}_i = \mathcal{B}_{1,i}$ for all $i\in \N.$

\begin{lemma}[Upper bound for $\|\Delta\|$]
\label{lemma:normadelta} 
Let $\mu$ be an invariant exponentially $\psi$-mixing measure on $\Omega^\mathbb N$ with mixing constants $T>0$ and $\sigma\in (0,1)$.  
Let $\Delta$ be  the    matrix of $\eta$-mixing coefficients. Then,
\[\|\Delta\|\leq 1+ 2{T\sigma}/{(1-\sigma)}.\]
\end{lemma}

\begin{proof}
Recall that we write $X^{\geq j}$ and $X^{\leq i}$ for $(X_j, X_{j+1}, \ldots )$
and $(X_1, \ldots, X_i)$ respectively.
Thus, for $w\in \Omega^i$, the probability that $X^{\leq i}= w$ is 
$\mu_i(w)$.
For $i,j\in \mathbb N$ such that $j\ge i$, take  $x, x'\in \Omega^ \N$  which differ only in coordinate $i$,
and 
 $A\in \mathcal{B}_{j,\infty}$ where   $j> i$.
Let us recall that  
\[
\eta_{i,j}=  \sup_{
    \begin{subarray}{c}
    x,x'\in \Omega^i,  
    \\ x[i]\neq x'[i]
    \end{subarray}
    }
    \sup_{A\in \mathcal{B}_{j,\infty}} 
 \left|\frac{\mu( C_i(x)\cap A)}{\mu_i(x)}-\frac{\mu( C_i(x')\cap A)}{\mu_i(x')}\right|.
\] 
By  our mixing property \eqref{eq:our-mixing}, 
\[
\left|\frac{\mu( C_i(x)\cap A)}{\mu_i(x)}-\mu(A)\right|+\left|\mu(A)-\frac{\mu( C_i(x')\cap A)}{\mu_i(x')}\right| \le 2\mu(A)T\sigma^{j-i}.\]
Then, each $\eta$-mixing coefficient  $\eta_{ij}\le 2T\sigma^{j-i}$.  

Let $J$ be the matrix
\begin{align*}
    J_{ij} = \begin{cases}
        1 & \textit{ if } j = i +1 \\
        0 & \textit{ otherwise},
    \end{cases}
\end{align*}
and  observe that $ \norm J \leq 1.$
By the triangular inequality and geometric series we get 
\begin{align*}
    \norm \Delta \leq 1 + T \sum_{i \geq 1 } \norm { (\sigma J)^i } \leq 1+2\frac{T\sigma}{1-\sigma}.
    \end{align*}
\end{proof}


\begin{lemma}[Two concentrations]\label{prop:concentration}
Let $\mu$ be an invariant and exponentially $\psi$-mixing measure on $\Omega^{\mathbb N}$. 
Let $k\in \mathbb N$ be  large enough and $j\ge 0$. Also, let $S\subset \mathbb R^{+}$ be a finite union of bounded intervals and denote by $\sup S$ the supremum of $S$. 
\begin{enumerate}

\item   
For  
 $\varphi_{k,S}:\Omega^\mathbb N \to \mathbb R$, 
$
\varphi_{k,S}(x)=\E_{\mu_k}[M_k^x(S)]
$ 
we have, for any $t>0$,
\begin{align*}
    \mu\pare{\{ x\in \Omega^\N: \abs{\varphi_{k,S}(x) - \E_\mu [\varphi_{k,S}] } \geq t \}} 
    \leq 
    2\exp\pare{
\frac{-t^2}
    {\|\Delta\|^28k^4 (\sup S) K^2 \rho^k}}.
\end{align*}

\item For  
 $\varphi_{k,j,S} : \Omega^{\mathbb N} \to [0,1]$,
$
\varphi_{k,j,S} (x)=\mu_k(\{w\in\Omega^k: M_k^x(w)(S)=j\})
$
we have, for any $t > 0$,
\begin{align*}
    \mu\pare{\{ x\in \Omega^\N: \abs{\varphi_{k,j,S}(x) - \E_\mu [\varphi_{k,j,S}] } \geq t \}} 
    \leq 
    2\exp\pare{{\frac{-t^2}{\|\Delta\|^28 k^2 (\sup {S})     K^2 \rho^k} }}.
\end{align*}

\end{enumerate}
The numbers   $\rho\in (0,1)$, and $K>0$
 are the contraction ratio and the constant 
given in \eqref{eq:rhoK}; and
$\Delta$ is the matrix of $\eta$-mixing coefficients associated with the mixing measure $\mu$.    
\end{lemma}

\begin{proof}
 Define the sequence of random variables $X_i:\Omega^\mathbb N\to \Omega$ as the usual projection onto the $i$-th coordinate of $x\in \Omega^\mathbb N$. 
Fix $k\in \mathbb N$ and fix a set $S$  which is a finite union of bounded intervals. 
For $i_0\in \N$ and 
$x,x'\in \Omega^\N$ such that $x$ and $x'$ differ only in 
the $i_0$-th coordinate, we define the set 
\[
D_{k,i_0,S}(x,x')=\{w\in \Omega^k: M_k^x(w)(S)\neq M_k^{x'}(w)(S)\}.
\]
 Since $x$ and $x'$ only differ in the coordinate $i_0$,  a word  $w\in D_{k,i_0,S}(x,x')$  is one of the following:  $ x[i_0 -k + i,i_0 + i)$ or $ x'[i_0 -k + i,i_0 + i)$ for $i\in \mathbb N$ and $1\le i\le k$       (with the convention that if $i< 0$, the word $x[i,i+k)$ is the empty word). Thus, there are, at most, $2k$ words in $D_{k,i_0,S}(x,x')$. 
For each    $w\in D_{k,i_0,S}(x,x')$, 
 \[
 \mu_k(w)\le \sup S/i_0
 \]
 because the sum in the definition of $M_k^x(w)(S)$ runs over the indexes $i$ in 
 \[\mathcal{J}_{w,S} = \set{i \in \N : i \mu_k(w) \in S}.\] 
 On the other hand, the measure of every word $w\in \Omega^k$ is upper bounded by the contraction ratio~$\rho$, $0<\rho<1$,
 \[
 \mu_k(w)\le K\rho^k.
 \] 

\emph{Point {\em 1}.}
Let $\varphi_{k,S}:\Omega^\N\to \R$, $
\varphi_{k,S}(x)=\E_{\mu_k}[M_k^x(S)].$ 
We prove that there exists  $\pare{c_{i}}_{i\ge 1}$ in $\ell^2$ such that, for every $N$, 
$\varphi_{k,S}(x)$ is $c$-Lipchitz.
In fact, 
\[
\left|\varphi_{k,S}(x)-\varphi_{k,S}(x')\right|\le    \mu_k(D_{k,i_0,S}(x,x'))\sup_{w\in D_{k,i_0,S}(x,x')} 
\left|M_k^x(w)(S)-M_k^{x'}(w)(S)\right|.
\]
If $x$ and $x'$ differ only in coordinate $i_0$, for each $w\in \Omega^k$,   $M_k^x(w)(S)$ and $ M_k^{x'}(w)(S)$ differ, at most, in the values of the $k$ indicators $I_{i+i_0-k}(\cdot,w)$ for $1\le i\le k$. This implies that $|M_k^x(w)(S)- M_k^{x'}(w)(S)|\le k$.
Define $c=\pare{c_i}_{i\in \mathbb N}$,
\[c_i= 2 k^2\min(K\rho^k,
(\sup S)/i).
\]
Then, 
\[\left|\varphi_{k,S}(x)-\varphi_{k,S}(x')\right|\le    c_i, 
\]
which means that  $\varphi_{k,S}$ is  $c$-Lipchitz. 
Let's see that  $\pare{c_i}_{i\ge 1}$ is in $\ell^2$:
\begin{align*}
   \|c\|^2= \sum_{i\geq 1} 
    c_i^2 
    & = \sum_{i \leq (\sup{S} ) /(K\rho^k)} c_i^2
    + \sum_{i > (\sup{S} )  /(K\rho^k) } c_i^2 \\
    & \leq {4}k^2 \pare{
        K^2 \sum_{i \leq 
        	(\sup{S} ) /(K\rho^k)} \rho^{2k}
        + \sup{S}^2 \sum_{i > (\sup{S} )/(K\rho^k) } 1/i^2
        }
        \\
    & \leq  {4} k^2 (\sup {S}) 
    K^2 \rho^k .
\end{align*}
The  function $\varphi_{k,S}$ is in $L^1(\mu)$ because of Lemma \ref{Le:expprod} and thus it
satisfies the assumptions of Proposition~\ref{lemma:KontRam}.
 Then, 
\[
\mu(\{x\in \Omega^\mathbb N: \left|
\varphi_{k,S}(x)-\E_{\mu}[\varphi_{k,S}] 
\right|
\ge t \}
\le 2\exp\pare{{\frac{-t^2}{2\|\Delta\|^2 \|c\|^2} }}.
\]

\emph{Point {\em 2.}}
Fix $j\geq 0$.
Let $\varphi_{k,j,S} : \Omega^{\mathbb N} \to [0,1]$,
$
\varphi_{k,j,S} (x)=\mu_k(\{w\in\Omega^k: M_k^x(w)(S)=j\})
$.
Define  $c=\pare{c_i}_{i\in \mathbb N}$ as 
\[c_i= 2k \min(K\rho^k,\sup S/i).\]
With the same argument as in Point 1, 
$c=\pare{c_i}_{i\in \mathbb N}$ is in $\ell^2$ and for any pair $x,x'\in \Omega^{\mathbb N}$ which only differ in the $i$th-coordinate we have, for each $i\ge 1$,
\[\left|\varphi_{k,j,S}(x)-\varphi_{k,j,S}(x')\right|\le    \mu_k(D_{k,i,S}(x,x'))\le c_i.
\]
Hence, $\varphi_{k,j,S}$ is  $c$-Lipchitz. 
It is also bounded and hence it is in $L^{1}(\mu)$. The assumptions of Proposition~\ref{lemma:KontRam} are fulfilled.
 Then,
\[\mu(\{x\in \Omega^\mathbb N: \left|\varphi_{k,j,S}(x)-\E_{\mu}[\varphi_{k,j,S}] \right|\ge t \}\le 2\exp\pare{{\frac{-t^2}{2\|\Delta\|^2 \|c\|^2} }}.\]
\end{proof}

The next lemma gives the wanted quenched result: it proves that for $\mu$-almost all $x\in\Omega^\N$, 
 the sequence of random measures $(M^x_k(.))_{k \geq 1}$ on $\R^+$ converges in distribution to a Poisson point process on $\R^+$ as $k$ goes to infinity.
 
 \begin{lemma}[The quenched result]\label{prop:quenched}
 Let $\mu$ be an invariant and exponentially $\psi$-mixing measure on $\Omega^{\mathbb N}$. 
Then for $\mu$-almost all $x\in\Omega^\N$, 
 the sequence of random measures $(M^x_k(.))_{k \geq 1}$ on $\R^+$ converges in distribution to a Poisson point process on $\R^+$ as $k$ goes to infinity. \end{lemma}

\begin{proof}
 To prove it  we apply Kallenberg's result stated in Proposition~\ref{kallenberg} for the measure $\mu$. 
 Fix $S\subseteq \R^+$ a finite union of intervals with rational endpoints.
Let  $j\geq 0$ be an integer.
We need to show that for $\mu$-almost all $x\in\Omega^\N$,
 \begin{enumerate}
\item $ \limsup_{k\to \infty} \E_{\mu_k}[M_k^x(S)]\le |S|.$

\item
$\lim_{k\to\infty}
\mu_k(\{w\in\Omega^k: M_k^x(w)(S)=j\}) = p(|S|,j).$
\end{enumerate}
We start with Point 1. 
Let
 $\varphi_{k,S} : \Omega^{\mathbb N} \to \R$,
$
\varphi_{k,S}(x)=\E_{\mu_k}[M_k^x(S)]
$ 
By Point 1 of  Lemma~\ref{prop:concentration} 
we have, for any $t>0$,
\begin{align*}
    \mu\pare{\{ x\in \Omega^\N: \abs{\varphi_{k,S}(x) - \E_\mu [\varphi_{k,S}] } \geq t \}} 
    \leq 
    2\exp\pare{
\frac{-t^2}
    {\|\Delta\|^28k^4 (\sup S) K^2 \rho^k}}.
\end{align*}
Taking $t_k = 1/k$,
using the bound for $\Delta$ given in Lemma~\ref{lemma:normadelta},
and using that $1/(k^6\rho^k)$ goes to infinity when $k$ goes to infinity, this series converges 
\[
\sum_{k\geq 1} \mu(\abs{\varphi_{k,S} - \E_\mu [\varphi_{k,S}]} > t_k) < \infty.
\]
By the Borel--Cantelli lemma
the limsup event
\begin{align*} 
\set{\,  x\in\Omega^\N: \, \abs{\varphi_{k,S}(x) - 
\E_\mu [\varphi_{k,S}(x)]} > t_{k} \, \text{for infinitely many }k } 
\end{align*}
has $\mu$-probability  zero. Then, there exists a set of $\mu$-measure $1$ (which may depend on $S$) where the difference 
\[\varphi_{k,S}(x) - 
\E_\mu [\varphi_{k,S}(x)]\]
goes to zero as $k$ goes to infinity .   
Since 
\[
\E_\mu [\varphi_{k,S}(x)]=
\E_{\mu\times \mu_k}[M_k^x(S)]
\]
and,  by Lemma \ref{Le:expprod}, 
\[
\E_{\mu\times \mu_k}[M_k^x(S)]
=|S|+O(\rho^k),
\]
it follows that  
$\varphi_{k,S}(x)=\E_{\mu_k}[M_k^x(S)]$
 converges to 
 $|S| $
as $k$ goes to infinity in a set of $\mu$-measure~$1$. 
Hence, $\limsup_{k\to\infty} \E_{\mu_k}[M_k^x(S)]\leq |S| $.

\medskip

Point 2. By Lemma~\ref{lemma:annealed}, for every $S$ that is a finite union of intervals with rational endpoints and for every integer $j\geq 0$,
\[
\mu\times \mu_k\pare{(x,w)\in \Omega^\mathbb N \times \Omega^k: M_k(x,w)(S)=j}\to
p(|S|,j)
\text{ as}\ k\to \infty.
\]
 Using that for any Borel set $A\subseteq \R^+$ the equality $\E_\mu[\mu_k(A)]=\mu\times {\mu_k}(A)$ holds,
 we have 
\[
 \E_\mu \corch{  \mu^k\pare{w\in \Omega^k:\,M^{x}_k(S)(w)=j
}}\to
p(|S|,j)
\text{ as}\ k\to \infty.
\]
Let
 $\varphi_{k,j,S} : \Omega^{\mathbb N} \to (0,1)$,
$
\varphi_{k,j,S}(x)=\mu_k(\{w\in\Omega^k: M_k^x(w)(S)=j\}).
$
By Point 2 of  Lemma~\ref{prop:concentration},
\[    \mu\pare{\{ x\in \Omega^\N: \abs{\varphi_{k,j,S}(x) - \E_\mu [\varphi_{k,j,S}] } \geq t \}} 
    \leq 
    2\exp\pare{{\frac{-t^2}{ \|\Delta\|^2 8 k^2 (\sup {S})     K^2 \rho^k} }}.
\]
Taking $t_k = 1/k$ and using that $1/(k^4\rho^k)$ 
goes to infinity when $k$ goes to infinity, this series converges 
\[
\sum_{k\geq 1} \mu(\abs{\varphi_{k,j,S} - \E_\mu [\varphi_{k,j,S}]} > t_k) < \infty.
\]
By the Borel--Cantelli lemma
the limsup event
\begin{align*} 
\set{\,  x\in\Omega^\N: \, \abs{\varphi_{k,j,S}(x) - 
\E_\mu [\varphi_{k,j,S}(x)]} > t_{k} \, \text{for infinitely many }k } 
\end{align*}
has $\mu$-probability  zero. Then, there exists a set of $\mu$-measure 1 (which may depend on $S$) where the difference 
\[\varphi_{k,j,S}(x) - 
\E_\mu [\varphi_{k,j,S}(x)]\]
goes to zero as $k$ goes to infinity .   
Since $\E_\mu \corch{  \mu^k\pare{w\in \Omega^k:\,M^{x}_k(S)(w)=j
}}$ converges to $
p(|S|,j)$ as $k$ goes to infinity in a set of $\mu$-measure 1, for $\mu$-almost all $x\in \Omega^\mathbb N$, 
\[
\abs{\mu_k\pare{\{w\in \Omega^k\,:\,M^{x}_k(S)(w)=j\} }- p(|S|,j) } \to 0\quad \text{as } k\to \infty. 
\]

There are countably many sets  $S$ that are finite union of intervals with rational endpoints,  and there are countably many integer values $j\geq 0$.
Then there are just  a countable number of each such exceptional sets (because 
 their   countable union  has also $\mu$-measure zero.
Using Proposition~\ref{kallenberg} we conclude that
for $\mu$-almost every $x\in\Omega^\N$, 
$
M^{x}_k(.)$ converges in distribution to a Poisson point process on $\R^+$.
\end{proof}
Lemma \ref{prop:quenched} shows that $\mu$-almost all $x \in \Omega^\N$ are Poisson-generic.
Theorem \ref{thm:main} is proved.

\section*{Acknowledgements}
We are
grateful to Inés Armendariz and Frédéric Paccaut. 
This work is supported by  grants  
CONICET PIP 11220210100220CO,
UBACyT 20020220100065BA, PICTO UNGS-00001 (2021), and UNGS-IDH 30/3373. 
\medskip

\noindent
Nicol\'as Álvarez \\
 ICC \\
 Universidad de Buenos Aires \&  CONICET  Argentina\\  {\tt  nialvarez@dc.uba.ar}
  
\medskip

\noindent
Ver\'onica Becher \\
 Departamento de  Computaci\'on, Facultad de Ciencias Exactas y Naturales \& ICC  \\
 Universidad de Buenos Aires \&  CONICET  Argentina\\  {\tt  vbecher@dc.uba.ar}
\medskip

\noindent Eda Cesaratto\\
Universidad Nacional de Gral. Sarmiento \& CONICET Argentina\\ {\tt
ecesaratto@campus.ungs.edu.ar}
\medskip

\noindent
Martín Mereb \\
 Departamento de Matemática, Facultad de Ciencias Exactas y Naturales \& IMAS \\
 Universidad de Buenos Aires \&  CONICET Argentina\\  {\tt  mmereb@gmail.com}
 \medskip

 \noindent
 Yuval Peres\\
 Beijing Institute of Mathematical Sciences and Applications (BIMSA)
 \\
 {\tt yuval@yuvalperes.com }
 \medskip

\noindent
Benjamin Weiss
\\
Institute of Mathematics, Hebrew University of Jerusalem \\
Jerusalem, Israel
 \\
 {\tt benjamin.weiss1@mail.huji.ac.il}
 \end{document}